\documentclass[a4paper,11pt]{article}

\usepackage{amsmath}
\usepackage{amssymb}
\usepackage{amsthm}
\usepackage{amsfonts}
\usepackage{mathtools}
\usepackage{ascmac}
\usepackage{spalign}
\usepackage{fancybox}
\usepackage[mathscr]{eucal}
\usepackage{footnpag}
\usepackage{siunitx}
\usepackage{multicol}
\usepackage{bbm}
\usepackage{framed}
\usepackage{here}

\usepackage{graphicx}
\usepackage{xcolor}
\usepackage{wrapfig}
\usepackage{float}

\theoremstyle{definition}
\newtheorem{theorem}{Theorem}[section]
\newtheorem*{theorem*}{Theorem}
\newtheorem{definition}{Definition}[section]
\newtheorem*{definition*}{Definition}
\newtheorem{proposition}{Proposition}[section]
\newtheorem*{proposition*}{Proposition}
\newtheorem{lemma}{Lemma}[section]
\newtheorem*{lemma*}{Lemma}
\newtheorem{corollary}{Corollary}[section]
\newtheorem*{corollary*}{Corollary}
\newtheorem{remark}{Remark}[section]
\newtheorem*{remark*}{Remark}

\title{Chebyshev Recurrence Structures for Reduced Spectral Functions of Cyclic Circulant Graphs}

\author{
Shunya Tamura\thanks{
Okegawa City Okegawa West Junior High School,
Saitama, 363-0027, Japan,
e-mail: \texttt{shunya.tamura059@gmail.com}
}
}

\date{}

\begin{document}
\maketitle

\begin{abstract}
Let \(S\) be a nonempty finite set of positive integers, let
\(q=\max S\), and let \(B_n(S)\), \(n>2q\), be the normalized product
sequence associated with the Chebyshev-type polynomial of the cyclic
circulant graph \(G_n(S)\).
When \(G_n(S)\) is connected, \(B_n(S)\) is both the normalized
spanning-tree number and a normalized special value of the
\emph{reduced spectral function}, a determinant-type function constructed
from the non-trivial adjacency spectrum.
Chebyshev root representations and the existence of linear recurrences for
fixed-step circulant spanning-tree sequences are known.
Starting from these representations, we explicitly construct a monic
annihilating polynomial
\(\mathcal H_S(X)\in\mathbb Z[X]\) of degree \(3^{q-1}\), which yields a
general upper bound for the recurrence order of \(B_n(S)\).
By collecting coincident exponential bases and accounting for possible
cancellations, we determine the minimal annihilating polynomial and give a
sufficient condition under which the minimal recurrence order is exactly
\(3^{q-1}\).
For \(S=\{1,2,3\}\) and \(S=\{1,3\}\), we explicitly derive the
corresponding ninth-degree annihilating polynomials and prove their minimality.
\end{abstract}



\section{Introduction}
\label{sec:introduction}

The number of spanning trees is a fundamental invariant in graph theory.
By Kirchhoff's Matrix--Tree Theorem, it is expressed as a product of the
nonzero Laplacian eigenvalues
\cite{Biggs1993,Kirchhoff1847}.
For regular graphs, this product may also be described in terms of the
adjacency spectrum.

Spanning-tree numbers of circulant graphs have been studied extensively.
Kleitman and Golden proved the classical identity
\[
\tau(C_n^2)=nF_n^2,
\]
where \(C_n^2\) denotes the square of the cycle \(C_n\), and \(F_n\) is
the \(n\)-th Fibonacci number
\cite{KleitmanGolden1975}.
Boesch and Prodinger developed formulas involving Chebyshev polynomials
\cite{BoeschProdinger1986}.
Zhang, Yong, and Golin established recurrence relations and Chebyshev
formulas for circulant and related graphs
\cite{ZhangYongGolin2000,ZhangYongGolin2005},
and further analyses were given in
\cite{AtajanYongInaba2006}.
Mednykh and Mednykh studied arithmetic and asymptotic properties of
circulant spanning-tree numbers and the rationality of their generating
functions
\cite{MednykhMednykh2019,MednykhMednykh2020}.

For connected circulant graphs with fixed jumps, earlier recurrence
formulations write the spanning-tree number in the form
\[
\tau(G_n(S))=n\,a_n(S)^2,
\]
where \(a_n(S)\) satisfies a linear recurrence
\cite{ZhangYongGolin2000,AtajanYongInaba2006}.
Thus, the normalization considered below satisfies
\[
B_n(S)=a_n(S)^2
\]
whenever \(G_n(S)\) is connected.
The root parameters used here are related to the standard Chebyshev
parameters by
\[
w_r=\frac{\phi_r}{2},
\]
while \(\rho_r\) is a root of
\[
z^2-2w_rz+1=0
\]
\cite{ZhangYongGolin2005,MednykhMednykh2019}.
Starting from these known representations, we explicitly construct a
monic annihilating polynomial with integer coefficients for \(B_n(S)\),
obtain a general bound for its recurrence order, and determine its minimal
annihilating polynomial after coincident exponential bases and
cancellations are taken into account.

For a connected \(d\)-regular graph \(G\) on \(n\) vertices, where
\(d\geq1\), let
\[
d=\mu_0,\mu_1,\ldots,\mu_{n-1}
\]
be its adjacency eigenvalues.
We define
\[
\widetilde{\zeta}_G(u)
=
\prod_{j=1}^{n-1}(1-u\mu_j)^{-1}.
\]
Equivalently,
\[
\widetilde{\zeta}_G(u)^{-1}
=
\frac{\det(I-uA(G))}{1-du},
\]
where \(A(G)\) denotes the adjacency matrix of \(G\).
We call \(\widetilde{\zeta}_G(u)\) the
\emph{reduced spectral function}.
This terminology is specific to the present paper;
\(\widetilde{\zeta}_G(u)^{-1}\) is the reciprocal form of the reduced
adjacency characteristic polynomial.
Its relation to standard graph and spectral zeta functions is discussed in
Section~\ref{sec:preliminaries}.
At \(u=1/d\),
\[
\widetilde{\zeta}_G\left(\frac{1}{d}\right)^{-1}
=
\frac{n\tau(G)}{d^{\,n-1}}.
\]

Let \(S\) be a nonempty finite set of positive integers and put
\[
q=\max S.
\]
For every integer \(n>2q\), consider the cyclic circulant graph
\[
G_n(S)
=
\operatorname{Cay}
\left(
\mathbb Z_n,\{\pm s:s\in S\}
\right).
\]
Associated with \(S\), define
\[
P_S(x)
=
2|S|
-
2\sum_{s\in S}
T_s\left(\frac{x}{2}\right),
\]
where \(T_s(x)\) is the Chebyshev polynomial of the first kind.
The polynomial \(P_S(x)\) belongs to \(\mathbb Z[x]\), has degree \(q\)
and leading coefficient \(-1\), and has \(x=2\) as a simple root.
Consequently,
\[
Q_S(x)
=
\frac{P_S(x)}{2-x}
\]
is a monic polynomial in \(\mathbb Z[x]\) of degree \(q-1\).

We study the normalized algebraic product
\[
B_n(S)
=
\frac{1}{n^2}
\prod_{j=1}^{n-1}
P_S\left(
2\cos\frac{2\pi j}{n}
\right),
\qquad
n>2q.
\]
This sequence is defined independently of whether \(G_n(S)\) is
connected.
When \(G_n(S)\) is connected, Kirchhoff's Matrix--Tree Theorem gives
\[
B_n(S)
=
\frac{\tau(G_n(S))}{n}
=
\frac{(2|S|)^{n-1}}{n^2}
\widetilde{\zeta}_{G_n(S)}
\left(
\frac{1}{2|S|}
\right)^{-1}.
\]

Put
\[
m=q-1,
\]
and let
\[
\phi_1,\ldots,\phi_m
\]
be the roots of \(Q_S(x)\), counted with multiplicities.
For each \(r=1,\ldots,m\), choose
\(\rho_r\in\mathbb C\setminus\{0\}\) satisfying
\[
\rho_r+\rho_r^{-1}=\phi_r.
\]
Then the Chebyshev root representation of \(B_n(S)\) takes the form
\[
B_n(S)
=
\prod_{r=1}^{m}
\left\{
(-1)^n
\frac{
\rho_r^n+\rho_r^{-n}-2
}{
2-\phi_r
}
\right\}.
\]
When \(m=0\), the empty product is understood to be \(1\).

For a sequence \(b=\{b_n\}\), let \(E\) denote the shift operator defined
by
\[
(Eb)_n=b_{n+1}.
\]
A polynomial \(F(X)\) is said to annihilate the tail sequence
\(\{b_n\}_{n>N}\) if
\[
F(E)b_n=0
\]
for every \(n>N\).
The root representation above leads to the following result.

\begin{theorem}[Main theorem]
\label{thm:intro-main}
Let \(S\) be a nonempty finite set of positive integers, let
\(q=\max S\), and put \(m=q-1\).
For
\[
\boldsymbol{\varepsilon}
=
(\varepsilon_1,\ldots,\varepsilon_m)
\in
\mathcal E_m
:=
\{-1,0,1\}^m,
\]
define
\[
\kappa(\boldsymbol{\varepsilon})
=
\left|
\left\{
r:\varepsilon_r=0
\right\}
\right|
\]
and
\[
\alpha_{\boldsymbol{\varepsilon}}
=
(-1)^m
\prod_{r=1}^{m}
\rho_r^{\varepsilon_r}.
\]
When \(m=0\), set
\[
\mathcal E_0=\{\varnothing\},
\qquad
\kappa(\varnothing)=0,
\qquad
\alpha_{\varnothing}=1.
\]
Then
\[
B_n(S)
=
\frac{1}{Q_S(2)}
\sum_{\boldsymbol{\varepsilon}\in\mathcal E_m}
(-2)^{\kappa(\boldsymbol{\varepsilon})}
\alpha_{\boldsymbol{\varepsilon}}^{\,n}
\qquad
(n>2q).
\]

Define
\[
\mathcal H_S(X)
=
\prod_{\boldsymbol{\varepsilon}\in\mathcal E_m}
\left(
X-\alpha_{\boldsymbol{\varepsilon}}
\right).
\]
Then
\[
\mathcal H_S(X)\in\mathbb Z[X],
\qquad
\deg\mathcal H_S=3^{q-1},
\]
and
\[
\mathcal H_S(E)B_n(S)=0
\qquad
(n>2q).
\]
Consequently, \(B_n(S)\) satisfies an integer-coefficient linear
recurrence of order at most \(3^{q-1}\).

Let
\[
\mathcal A_S
=
\left\{
\alpha_{\boldsymbol{\varepsilon}}:
\boldsymbol{\varepsilon}\in\mathcal E_m
\right\}
\]
be the set of distinct exponential bases, and, for each
\(\gamma\in\mathcal A_S\), define
\[
C_\gamma
=
\sum_{\substack{
\boldsymbol{\varepsilon}\in\mathcal E_m\\
\alpha_{\boldsymbol{\varepsilon}}=\gamma
}}
(-2)^{\kappa(\boldsymbol{\varepsilon})}.
\]
Put
\[
\mathcal A_S^{*}
=
\left\{
\gamma\in\mathcal A_S:
C_\gamma\neq0
\right\}.
\]
Then the minimal monic polynomial \(F(X)\in\mathbb C[X]\) satisfying
\[
F(E)B_n(S)=0
\qquad
(n>2q)
\]
is
\[
\mathcal M_S(X)
=
\prod_{\gamma\in\mathcal A_S^{*}}
(X-\gamma).
\]
Moreover,
\[
\mathcal M_S(X)\in\mathbb Z[X].
\]
In particular, if the \(3^{q-1}\) indexed values
\(\alpha_{\boldsymbol{\varepsilon}}\) are pairwise distinct, then
\[
\mathcal M_S(X)=\mathcal H_S(X),
\]
and the minimal recurrence order is exactly \(3^{q-1}\).
\end{theorem}

The cases \(S=\{1\}\) and \(S=\{1,2\}\) are included as classical
illustrations.
For \(S=\{1,2,3\}\) and \(S=\{1,3\}\), we explicitly derive the
corresponding ninth-degree annihilating polynomials and prove their
minimality.
The latter example lies outside the cycle-power family.

Section~\ref{sec:preliminaries} recalls the required spectral and
Chebyshev identities.
Sections~\ref{sec:circulant} and \ref{sec:special-value} introduce the
circulant graphs and the normalized product sequence.
Section~\ref{sec:root-recurrence} proves the main results.
Sections~\ref{sec:examples} and \ref{sec:verification} present explicit
examples and exact algebraic checks.
Section~\ref{sec:conclusion} concludes the paper.

\section{Preliminaries}
\label{sec:preliminaries}

We recall the spectral and Chebyshev identities used in the subsequent
sections.
For standard facts on spectral graph theory, we refer to
\cite{Biggs1993}.
Unless otherwise stated, all graphs are finite, simple, connected, and
undirected.

Let \(G\) be a \(d\)-regular graph on \(n\) vertices, where \(d\geq1\).
Write
\[
A=A(G),
\qquad
L=dI-A
\]
for its adjacency and Laplacian matrices.
Let
\[
d=\mu_0,\mu_1,\ldots,\mu_{n-1}
\]
be the adjacency eigenvalues.
Since \(G\) is connected, the eigenvalue \(\mu_0=d\) is simple, and the
Laplacian eigenvalues are
\[
0=\lambda_0,\lambda_1,\ldots,\lambda_{n-1},
\qquad
\lambda_j=d-\mu_j.
\]

\subsection{Reduced spectral function}
\label{subsec:reduced-spectral-function}

\begin{definition}
\label{def:reduced-spectral-function}
The \emph{reduced spectral function} of \(G\) is defined by
\[
\widetilde{\zeta}_G(u)
=
\prod_{j=1}^{n-1}
(1-u\mu_j)^{-1}.
\]
\end{definition}

Let
\[
\chi_G(t)
=
\det(tI-A)
\]
be the adjacency characteristic polynomial, and put
\[
\chi_G^{\mathrm{red}}(t)
=
\frac{\chi_G(t)}{t-d}.
\]
Then
\[
\begin{aligned}
\widetilde{\zeta}_G(u)^{-1}
&=
\prod_{j=1}^{n-1}(1-u\mu_j)
\\
&=
u^{n-1}
\chi_G^{\mathrm{red}}(u^{-1})
\\
&=
\frac{\det(I-uA)}{1-du}.
\end{aligned}
\]
Thus, \(\widetilde{\zeta}_G(u)^{-1}\) is the reciprocal form of the
reduced adjacency characteristic polynomial.

\begin{remark}
\label{rem:reduced-function-terminology}
The term \emph{reduced spectral function} is specific to the present
paper.
It should not be confused with the Ihara zeta function
\cite{Terras2010}
or with a spectral zeta function of the form
\[
\sum_{\lambda\neq0}\lambda^{-s}
\]
\cite{Kirsten2001}.
Here it is used as a determinant-type normalization of the non-trivial
adjacency spectrum.

Connectedness ensures that the eigenvalue \(d\) is simple.
For a disconnected \(d\)-regular graph, removing only one factor
\(1-du\) would leave an additional zero at \(u=1/d\).
The algebraic product sequence introduced later is therefore defined
independently of this connectedness assumption.
\end{remark}

\begin{proposition}
\label{prop:reduced-special-value}
Let \(G\) be a connected \(d\)-regular graph on \(n\) vertices, where
\(d\geq1\).
Then
\[
\widetilde{\zeta}_G\left(\frac{1}{d}\right)^{-1}
=
\frac{1}{d^{\,n-1}}
\prod_{j=1}^{n-1}\lambda_j.
\]
\end{proposition}

\begin{proof}
Since
\[
1-\frac{\mu_j}{d}
=
\frac{d-\mu_j}{d}
=
\frac{\lambda_j}{d},
\]
the result follows by taking the product over
\(j=1,\ldots,n-1\).
\end{proof}

Kirchhoff's Matrix--Tree Theorem gives
\[
\tau(G)
=
\frac{1}{n}
\prod_{j=1}^{n-1}\lambda_j.
\]
Hence Proposition \ref{prop:reduced-special-value} immediately yields the
following consequence.

\begin{corollary}
\label{cor:reduced-spanning-tree}
Let \(G\) be a connected \(d\)-regular graph on \(n\) vertices, where
\(d\geq1\).
Then
\[
\widetilde{\zeta}_G\left(\frac{1}{d}\right)^{-1}
=
\frac{n\tau(G)}{d^{\,n-1}}.
\]
\end{corollary}

\begin{remark}
\label{rem:special-value-information}
\label{rem:full-function-information}
The special value in Corollary
\ref{cor:reduced-spanning-tree} is a fixed normalization of the
spanning-tree number.
The full function, however, also determines other spectral quantities.
The following proposition gives one such example.
\end{remark}

Recall that the Kirchhoff index has the spectral representation
\[
\operatorname{Kf}(G)
=
n\sum_{j=1}^{n-1}
\frac{1}{\lambda_j}.
\]

\begin{proposition}
\label{prop:logarithmic-derivative}
Let \(G\) be a connected \(d\)-regular graph on \(n\) vertices, where
\(d\geq1\).
Then
\[
\operatorname{Kf}(G)
=
\frac{n}{d^2}
\left\{
\left.
\frac{\widetilde{\zeta}_G'(u)}
     {\widetilde{\zeta}_G(u)}
\right|_{u=1/d}
+
d(n-1)
\right\}.
\]
\end{proposition}

\begin{proof}
Taking the logarithmic derivative gives
\[
\frac{\widetilde{\zeta}_G'(u)}
     {\widetilde{\zeta}_G(u)}
=
\sum_{j=1}^{n-1}
\frac{\mu_j}{1-u\mu_j}.
\]
Since \(\mu_j=d-\lambda_j\), we obtain
\[
\begin{aligned}
\left.
\frac{\widetilde{\zeta}_G'(u)}
     {\widetilde{\zeta}_G(u)}
\right|_{u=1/d}
&=
d\sum_{j=1}^{n-1}
\frac{d-\lambda_j}{\lambda_j}
\\
&=
d^2\sum_{j=1}^{n-1}
\frac{1}{\lambda_j}
-
d(n-1).
\end{aligned}
\]
The assertion follows from the spectral formula for
\(\operatorname{Kf}(G)\).
\end{proof}

\subsection{Chebyshev polynomials}
\label{subsec:chebyshev-polynomials}

The Chebyshev polynomials of the first kind are defined by
\[
T_r(\cos\theta)
=
\cos(r\theta).
\]
They satisfy
\[
T_0(x)=1,
\qquad
T_1(x)=x,
\qquad
T_{r+1}(x)
=
2xT_r(x)-T_{r-1}(x).
\]
For their use in spanning-tree formulas, see
\cite{BoeschProdinger1986,ZhangYongGolin2005}.

\begin{lemma}
\label{lem:chebyshev-cos}
If
\[
x=2\cos\theta,
\]
then, for every integer \(r\geq0\),
\[
2\cos(r\theta)
=
2T_r\left(\frac{x}{2}\right).
\]
\end{lemma}

\begin{proof}
This follows directly from
\[
\frac{x}{2}=\cos\theta
\]
and the defining identity of \(T_r\).
\end{proof}

We shall also use
\[
T_r'(1)=r^2
\qquad
(r\geq1).
\]
Indeed,
\[
T_r'(x)
=
rU_{r-1}(x),
\]
where \(U_{r-1}(x)\) is the Chebyshev polynomial of the second kind, and
\[
U_{r-1}(1)=r.
\]

\section{Cyclic circulant graphs and Chebyshev polynomials}
\label{sec:circulant}

We fix the notation for cyclic circulant graphs and record the Fourier
formulas used below.
For related spectral and Chebyshev formulas, see
\cite{BoeschProdinger1986,ZhangYongGolin2005,MednykhMednykh2019}.

Let \(S\) be a nonempty finite set of positive integers and put
\[
q=\max S.
\]
Throughout the paper, we assume that
\[
n>2q.
\]
Then the residues
\[
\{\pm s:s\in S\}
\]
are nonzero and pairwise distinct modulo \(n\).

\begin{definition}
\label{def:cyclic-circulant}
The cyclic circulant graph associated with \(S\) is
\[
G_n(S)
=
\operatorname{Cay}
\left(
\mathbb Z_n,\{\pm s:s\in S\}
\right).
\]
Thus,
\[
V(G_n(S))=\mathbb Z_n,
\]
and two vertices \(a,b\in\mathbb Z_n\) are adjacent if
\[
a-b\equiv\pm s\pmod n
\]
for some \(s\in S\).
\end{definition}

The graph \(G_n(S)\) is a simple \(2|S|\)-regular graph.
We write
\[
d=2|S|.
\]
Moreover, it is connected if and only if
\[
\gcd(n,s_1,\ldots,s_t)=1,
\]
where
\[
S=\{s_1,\ldots,s_t\}.
\]
The algebraic product sequences introduced later are defined without
assuming connectedness.
Connectedness will be imposed only when spanning-tree numbers or the
reduced spectral function are considered.

\begin{remark}
\label{rem:small-n}
The condition \(n>2q\) ensures that one fixed step set \(S\) determines a
simple \(2|S|\)-regular graph for every admissible \(n\).
When \(n\leq2q\), distinct elements of
\[
\{\pm s:s\in S\}
\]
may coincide modulo \(n\), and a step may satisfy
\[
s\equiv-s\pmod n.
\]
If the steps are retained with multiplicities, the Fourier formulas remain
valid for the corresponding weighted Cayley multigraph.
For the underlying simple graph, however, coincident residues must be
merged, so both the degree and the polynomial \(P_S(x)\) must be modified.
We therefore retain the assumption \(n>2q\).
\end{remark}

\subsection{Adjacency spectrum}
\label{subsec:adjacency-spectrum}

\begin{proposition}
\label{prop:adjacency-spectrum}
Let
\[
\omega_n=e^{2\pi i/n}.
\]
The adjacency eigenvalues of \(G_n(S)\), indexed by the Fourier modes, are
\[
\mu_j
=
\sum_{s\in S}
\left(
\omega_n^{js}+\omega_n^{-js}
\right)
=
2\sum_{s\in S}
\cos\frac{2\pi js}{n},
\qquad
j=0,1,\ldots,n-1.
\]
In particular,
\[
\mu_0=2|S|=d.
\]
\end{proposition}

\begin{proof}
For \(j=0,1,\ldots,n-1\), define
\[
v_j(a)=\omega_n^{ja},
\qquad
a\in\mathbb Z_n.
\]
Then
\[
\begin{aligned}
(Av_j)(a)
&=
\sum_{s\in S}
\left\{
v_j(a+s)+v_j(a-s)
\right\}
\\
&=
\left\{
\sum_{s\in S}
\left(
\omega_n^{js}+\omega_n^{-js}
\right)
\right\}
v_j(a).
\end{aligned}
\]
Since
\[
v_0,v_1,\ldots,v_{n-1}
\]
form a basis of \(\mathbb C^{\mathbb Z_n}\), this gives the full adjacency
spectrum.
The cosine form follows from
\[
\omega_n^{js}+\omega_n^{-js}
=
2\cos\frac{2\pi js}{n}.
\]
\end{proof}

\subsection{Laplacian spectrum and the Chebyshev-type polynomial}
\label{subsec:laplacian-chebyshev}

\begin{definition}
\label{def:PS}
The Chebyshev-type polynomial associated with \(S\) is
\[
P_S(x)
=
2|S|
-
2\sum_{s\in S}
T_s\left(\frac{x}{2}\right).
\]
\end{definition}

\begin{proposition}
\label{prop:PS-basic}
Let \(q=\max S\).
Then
\[
P_S(x)\in\mathbb Z[x],
\qquad
\deg P_S=q,
\]
and the leading coefficient of \(P_S(x)\) is \(-1\).
Moreover,
\[
P_S(2)=0.
\]
\end{proposition}

\begin{proof}
For \(r\geq0\), put
\[
C_r(x)
=
2T_r\left(\frac{x}{2}\right).
\]
Then
\[
C_0(x)=2,
\qquad
C_1(x)=x,
\qquad
C_{r+1}(x)
=
xC_r(x)-C_{r-1}(x).
\]
It follows inductively that, for every \(r\geq1\),
\[
C_r(x)\in\mathbb Z[x]
\]
is monic of degree \(r\).

Since \(q\) is the largest element of \(S\), the term
\[
-C_q(x)
=
-2T_q\left(\frac{x}{2}\right)
\]
is the unique term of degree \(q\) in \(P_S(x)\).
Hence \(P_S(x)\) has degree \(q\) and leading coefficient \(-1\).

Finally,
\[
T_s(1)=1
\]
for every \(s\in S\), and therefore
\[
P_S(2)
=
2|S|
-
2\sum_{s\in S}T_s(1)
=
0.
\]
\end{proof}

\begin{proposition}
\label{thm:laplacian-chebyshev}
The Laplacian eigenvalues of \(G_n(S)\), indexed by the Fourier modes, are
\[
\lambda_j
=
P_S\left(
2\cos\frac{2\pi j}{n}
\right),
\qquad
j=0,1,\ldots,n-1.
\]
\end{proposition}

\begin{proof}
Since \(G_n(S)\) is \(d=2|S|\)-regular,
\[
\lambda_j=d-\mu_j.
\]
By Proposition \ref{prop:adjacency-spectrum},
\[
\lambda_j
=
2|S|
-
2\sum_{s\in S}
\cos\frac{2\pi js}{n}.
\]
Put
\[
x_j
=
2\cos\frac{2\pi j}{n}.
\]
Lemma \ref{lem:chebyshev-cos} gives
\[
2\cos\frac{2\pi js}{n}
=
2T_s\left(\frac{x_j}{2}\right),
\]
and hence
\[
\lambda_j
=
2|S|
-
2\sum_{s\in S}
T_s\left(\frac{x_j}{2}\right)
=
P_S(x_j).
\]
\end{proof}

\begin{remark}
\label{rem:standard-fourier-role}
The eigenvalues above are indexed by the Fourier modes
\(j=0,\ldots,n-1\).
Thus, if distinct Fourier modes produce the same numerical eigenvalue, that
value is retained with its spectral multiplicity.
This convention is used in all subsequent products over \(j\).
\end{remark}

\section{Special values of the reduced spectral function}
\label{sec:special-value}

We introduce two algebraic product sequences associated with the fixed
step set \(S\).
They are defined for every \(n>2q\), independently of whether the
corresponding circulant graph is connected.

\begin{definition}
\label{def:Mn}
For every integer \(n>2q\), define
\[
M_n(S)
=
\prod_{j=1}^{n-1}
P_S\left(
2\cos\frac{2\pi j}{n}
\right).
\]
\end{definition}

By Proposition \ref{thm:laplacian-chebyshev}, \(M_n(S)\) is the product of
the Laplacian eigenvalues indexed by the nonzero Fourier modes.
Consequently, if \(G_n(S)\) is disconnected, then
\[
M_n(S)=0.
\]

\begin{proposition}
\label{prop:Mn-spanning-tree}
If \(G_n(S)\) is connected, then
\[
M_n(S)
=
n\tau(G_n(S))
=
(2|S|)^{n-1}
\widetilde{\zeta}_{G_n(S)}
\left(
\frac{1}{2|S|}
\right)^{-1}.
\]
\end{proposition}

\begin{proof}
When \(G_n(S)\) is connected, \(M_n(S)\) is the product of all nonzero
Laplacian eigenvalues.
The first equality follows from Kirchhoff's Matrix--Tree Theorem, and the
second from Corollary \ref{cor:reduced-spanning-tree}.
\end{proof}

\begin{definition}
\label{def:Bn}
For every integer \(n>2q\), define
\[
B_n(S)
=
\frac{M_n(S)}{n^2}.
\]
\end{definition}

\begin{corollary}
\label{cor:Bn-interpretation}
If \(G_n(S)\) is connected, then
\[
B_n(S)
=
\frac{\tau(G_n(S))}{n}
=
\frac{(2|S|)^{n-1}}{n^2}
\widetilde{\zeta}_{G_n(S)}
\left(
\frac{1}{2|S|}
\right)^{-1}.
\]
\end{corollary}

\begin{proof}
This follows immediately from Definition \ref{def:Bn} and Proposition
\ref{prop:Mn-spanning-tree}.
\end{proof}

\begin{remark}
\label{rem:Bn-connectivity}
\label{rem:role-of-Mn}
The sequences \(M_n(S)\) and \(B_n(S)\) are defined algebraically for all
\(n>2q\).
If
\[
\gcd(n,s_1,\ldots,s_t)>1,
\]
then \(G_n(S)\) is disconnected and
\[
M_n(S)=B_n(S)=0.
\]
If the graph is connected, Corollary \ref{cor:Bn-interpretation} gives
their spanning-tree and reduced-spectral-function interpretations.

The recurrence analysis below concerns the full algebraic sequence
\[
\{B_n(S)\}_{n>2q},
\]
including the indices for which the corresponding graph is disconnected.
\end{remark}

\section{Root expression and linear recurrence structure}
\label{sec:root-recurrence}

We factor the algebraic product \(M_n(S)\) in terms of the roots of a
polynomial obtained from \(P_S(x)\), and then use the resulting exponential
representation to construct an integer annihilating polynomial for
\(B_n(S)\).

Let
\[
q=\max S.
\]

\subsection{The reduced polynomial}
\label{subsec:reduced-polynomial}

Since \(P_S(2)=0\), the polynomial \(2-x\) divides \(P_S(x)\) in
\(\mathbb Z[x]\).

\begin{definition}
\label{def:QS}
Define
\[
Q_S(x)
=
\frac{P_S(x)}{2-x}.
\]
\end{definition}

\begin{proposition}
\label{prop:QS-basic}
The polynomial \(Q_S(x)\) is monic, belongs to \(\mathbb Z[x]\), and has
degree \(q-1\).
Moreover,
\[
Q_S(2)
=
\sum_{s\in S}s^2
\neq0.
\]
In particular, \(x=2\) is a simple root of \(P_S(x)\).
\end{proposition}

\begin{proof}
By Proposition \ref{prop:PS-basic},
\[
P_S(x)\in\mathbb Z[x],
\qquad
P_S(2)=0.
\]
Hence division by \(2-x\) gives a polynomial in \(\mathbb Z[x]\).
Since both \(P_S(x)\) and \(2-x\) have leading coefficient \(-1\),
\(Q_S(x)\) is monic and
\[
\deg Q_S=q-1.
\]

Differentiating
\[
P_S(x)
=
2|S|
-
2\sum_{s\in S}
T_s\left(\frac{x}{2}\right)
\]
gives
\[
P_S'(x)
=
-
\sum_{s\in S}
T_s'\left(\frac{x}{2}\right).
\]
Using
\[
T_s'(1)=s^2,
\]
we obtain
\[
P_S'(2)
=
-
\sum_{s\in S}s^2.
\]
On the other hand, differentiating
\[
P_S(x)=(2-x)Q_S(x)
\]
and substituting \(x=2\) gives
\[
P_S'(2)=-Q_S(2).
\]
Therefore,
\[
Q_S(2)
=
\sum_{s\in S}s^2
\neq0.
\]
\end{proof}

Let
\[
\phi_1,\ldots,\phi_{q-1}
\]
be the roots of \(Q_S(x)\), counted with multiplicities.
Since \(Q_S(x)\) is monic,
\[
Q_S(x)
=
\prod_{r=1}^{q-1}(x-\phi_r),
\]
and hence
\[
P_S(x)
=
(2-x)
\prod_{r=1}^{q-1}(x-\phi_r).
\]
Proposition \ref{prop:QS-basic} ensures that
\[
\phi_r\neq2
\qquad
(r=1,\ldots,q-1).
\]

\subsection{Chebyshev root and exponential representations}
\label{subsec:root-product}

We use the following Chebyshev product identity, which underlies the
circulant spanning-tree formulas in
\cite{BoeschProdinger1986,ZhangYongGolin2005}.

\begin{lemma}
\label{lem:chebyshev-product}
For every \(\phi\neq2\),
\[
\prod_{j=1}^{n-1}
\left(
2\cos\frac{2\pi j}{n}-\phi
\right)
=
(-1)^n
\frac{
2T_n\left(\frac{\phi}{2}\right)-2
}{
2-\phi
}.
\]
Moreover,
\[
\prod_{j=1}^{n-1}
\left(
2-2\cos\frac{2\pi j}{n}
\right)
=
n^2.
\]
\end{lemma}

\begin{proof}
The polynomial
\[
2\left\{
T_n\left(\frac{x}{2}\right)-1
\right\}
\]
is monic of degree \(n\), and its roots, counted with multiplicities, are
\[
2\cos\frac{2\pi j}{n},
\qquad
j=0,1,\ldots,n-1.
\]
Here the list is interpreted as a multiset; for example,
\[
2\cos\frac{2\pi j}{n}
=
2\cos\frac{2\pi(n-j)}{n}.
\]
Therefore,
\[
\prod_{j=0}^{n-1}
\left(
x-2\cos\frac{2\pi j}{n}
\right)
=
2\left\{
T_n\left(\frac{x}{2}\right)-1
\right\}.
\]
Removing the factor corresponding to \(j=0\) gives
\[
\prod_{j=1}^{n-1}
\left(
x-2\cos\frac{2\pi j}{n}
\right)
=
\frac{
2\left\{
T_n\left(\frac{x}{2}\right)-1
\right\}
}{
x-2
}.
\]
Substituting \(x=\phi\) and reversing the signs of the factors proves the
first identity.

Taking the limit as \(x\to2\), we obtain
\[
\begin{aligned}
\prod_{j=1}^{n-1}
\left(
2-2\cos\frac{2\pi j}{n}
\right)
&=
\lim_{x\to2}
\frac{
2\left\{
T_n\left(\frac{x}{2}\right)-1
\right\}
}{
x-2
}
\\
&=
T_n'(1)
\\
&=
n^2.
\end{aligned}
\]
\end{proof}

\begin{proposition}
\label{thm:root-product}
For every integer \(n>2q\),
\[
M_n(S)
=
n^2
\prod_{r=1}^{q-1}
\left\{
(-1)^n
\frac{
2T_n\left(\frac{\phi_r}{2}\right)-2
}{
2-\phi_r
}
\right\}.
\]
Equivalently,
\[
B_n(S)
=
\prod_{r=1}^{q-1}
\left\{
(-1)^n
\frac{
2T_n\left(\frac{\phi_r}{2}\right)-2
}{
2-\phi_r
}
\right\}.
\]
When \(q=1\), the products over \(r\) are understood to be empty and hence
equal to \(1\).
\end{proposition}

\begin{proof}
Using
\[
P_S(x)
=
(2-x)
\prod_{r=1}^{q-1}(x-\phi_r),
\]
we have
\[
\begin{aligned}
M_n(S)
&=
\prod_{j=1}^{n-1}
\left(
2-2\cos\frac{2\pi j}{n}
\right)
\\
&\qquad\times
\prod_{r=1}^{q-1}
\prod_{j=1}^{n-1}
\left(
2\cos\frac{2\pi j}{n}-\phi_r
\right).
\end{aligned}
\]
The formula for \(M_n(S)\) follows from Lemma
\ref{lem:chebyshev-product}, and division by \(n^2\) gives the formula for
\(B_n(S)\).
\end{proof}

For each \(r=1,\ldots,q-1\), choose
\(\rho_r\in\mathbb C\setminus\{0\}\) satisfying
\[
\rho_r+\rho_r^{-1}=\phi_r.
\]
Then
\[
2T_n\left(\frac{\phi_r}{2}\right)
=
\rho_r^n+\rho_r^{-n}.
\]

\begin{corollary}
\label{thm:exponential-product}
For every integer \(n>2q\),
\[
M_n(S)
=
n^2
\prod_{r=1}^{q-1}
\left\{
(-1)^n
\frac{
\rho_r^n+\rho_r^{-n}-2
}{
2-\phi_r
}
\right\}.
\]
Equivalently,
\[
B_n(S)
=
\prod_{r=1}^{q-1}
\frac{
(-\rho_r)^n+(-\rho_r^{-1})^n-2(-1)^n
}{
2-\phi_r
}.
\]
\end{corollary}

\begin{proof}
Substitute
\[
2T_n\left(\frac{\phi_r}{2}\right)
=
\rho_r^n+\rho_r^{-n}
\]
into Proposition \ref{thm:root-product}.
The second formula follows from
\[
(-1)^n
\left(
\rho_r^n+\rho_r^{-n}-2
\right)
=
(-\rho_r)^n
+
(-\rho_r^{-1})^n
-
2(-1)^n.
\]
\end{proof}

\begin{remark}
\label{rem:exceptional-root-minus-two}
The value \(\phi_r=2\) cannot occur by Proposition
\ref{prop:QS-basic}.
If \(\phi_r=-2\), then
\[
\rho_r=\rho_r^{-1}=-1,
\]
and the corresponding factor in Corollary
\ref{thm:exponential-product} is
\[
\frac{1-(-1)^n}{2}.
\]
It vanishes when \(n\) is even.
In this case,
\[
-2
=
2\cos\pi
\]
is a Fourier point and produces an additional zero Laplacian eigenvalue,
consistently with
\[
M_n(S)=B_n(S)=0.
\]
Repeated roots of \(Q_S(x)\) and coincident exponential bases are allowed
in the construction below.
\end{remark}

\subsection{Integer annihilating polynomials}
\label{subsec:integer-annihilator}

Put
\[
m=q-1
\]
and define
\[
\mathcal E_m
=
\{-1,0,1\}^m.
\]
For
\[
\boldsymbol{\varepsilon}
=
(\varepsilon_1,\ldots,\varepsilon_m)
\in\mathcal E_m,
\]
let
\[
\kappa(\boldsymbol{\varepsilon})
=
\left|
\left\{
r:\varepsilon_r=0
\right\}
\right|
\]
and
\[
\alpha_{\boldsymbol{\varepsilon}}
=
(-1)^m
\prod_{r=1}^{m}
\rho_r^{\varepsilon_r}.
\]
The choices
\[
\varepsilon_r=1,
\qquad
\varepsilon_r=-1,
\qquad
\varepsilon_r=0
\]
correspond respectively to
\[
-\rho_r,
\qquad
-\rho_r^{-1},
\qquad
-1.
\]

When \(m=0\), we set
\[
\mathcal E_0=\{\varnothing\},
\qquad
\kappa(\varnothing)=0,
\qquad
\alpha_{\varnothing}=1.
\]

Since \(Q_S(x)\) is monic,
\[
\prod_{r=1}^{m}(2-\phi_r)
=
Q_S(2)
=
\sum_{s\in S}s^2.
\]
Expanding the product in Corollary
\ref{thm:exponential-product} gives
\begin{equation}
\label{eq:Bn-expanded}
B_n(S)
=
\frac{1}{Q_S(2)}
\sum_{\boldsymbol{\varepsilon}\in\mathcal E_m}
(-2)^{\kappa(\boldsymbol{\varepsilon})}
\alpha_{\boldsymbol{\varepsilon}}^{\,n}.
\end{equation}

We regard
\[
\mathfrak A_S
=
\left(
\alpha_{\boldsymbol{\varepsilon}}
\right)_{\boldsymbol{\varepsilon}\in\mathcal E_m}
\]
as a multiset indexed by \(\mathcal E_m\), so equal values are retained
with their multiplicities.
Its underlying set of distinct values is
\[
\mathcal A_S
=
\left\{
\alpha_{\boldsymbol{\varepsilon}}:
\boldsymbol{\varepsilon}\in\mathcal E_m
\right\}.
\]
Hence
\[
|\mathcal A_S|
\leq
3^m
=
3^{q-1}.
\]

Let \(E\) denote the shift operator
\[
Eb_n=b_{n+1}.
\]
A polynomial \(F(X)\) is said to annihilate a tail sequence
\(\{b_n\}_{n>N}\) if
\[
F(E)b_n=0
\]
for every \(n>N\).

\begin{lemma}
\label{lem:minimal-exponential-recurrence}
Let \(\Gamma\) be a finite set of distinct nonzero complex numbers, and
suppose that
\[
b_n
=
\sum_{\gamma\in\Gamma}
c_\gamma\gamma^n,
\qquad
c_\gamma\neq0,
\]
for every \(n>N\).
Then the minimal monic annihilating polynomial of
\(\{b_n\}_{n>N}\) over \(\mathbb C\) is
\[
\prod_{\gamma\in\Gamma}(X-\gamma).
\]
\end{lemma}

\begin{proof}
The displayed polynomial annihilates every sequence \(\gamma^n\), and
therefore annihilates \(b_n\).

Conversely, suppose that \(F(E)b_n=0\).
Then
\[
\sum_{\gamma\in\Gamma}
c_\gamma F(\gamma)\gamma^n
=
0
\]
for every \(n>N\).
Taking \(|\Gamma|\) consecutive values of \(n\) gives a Vandermonde system.
Since the elements of \(\Gamma\) are distinct and nonzero, its determinant
is nonzero.
Thus,
\[
c_\gamma F(\gamma)=0
\]
for every \(\gamma\in\Gamma\), and hence
\[
F(\gamma)=0
\]
for every \(\gamma\in\Gamma\).
Therefore, \(F(X)\) is divisible by
\[
\prod_{\gamma\in\Gamma}(X-\gamma).
\]
\end{proof}

\begin{theorem}
\label{thm:integer-annihilating-polynomial}
Let \(S\) be a nonempty finite set of positive integers, let
\(q=\max S\), and put \(m=q-1\).
Define
\[
\mathcal H_S(X)
=
\prod_{\boldsymbol{\varepsilon}\in\mathcal E_m}
\left(
X-\alpha_{\boldsymbol{\varepsilon}}
\right).
\]
Then
\[
\mathcal H_S(X)\in\mathbb Z[X],
\qquad
\deg\mathcal H_S
=
3^{q-1},
\]
and \(\mathcal H_S(X)\) annihilates the tail sequence
\[
\{B_n(S)\}_{n>2q}.
\]
Consequently, \(B_n(S)\) satisfies an integer-coefficient linear recurrence
of order at most
\[
3^{q-1}.
\]
\end{theorem}

\begin{proof}
Since \(Q_S(x)\) is monic and belongs to \(\mathbb Z[x]\), every root
\(\phi_r\) is an algebraic integer.
Moreover, \(\rho_r\) satisfies
\[
z^2-\phi_rz+1=0.
\]
Thus, every \(\rho_r\), and hence every
\(\alpha_{\boldsymbol{\varepsilon}}\), is an algebraic integer.

Let
\[
\sigma
\in
\operatorname{Gal}
\left(
\overline{\mathbb Q}/\mathbb Q
\right).
\]
Since \(Q_S(x)\in\mathbb Z[x]\), there is a permutation
\(\pi_\sigma\) of \(\{1,\ldots,m\}\) such that
\[
\sigma(\phi_r)
=
\phi_{\pi_\sigma(r)}.
\]
For each \(r\), there is a sign
\[
\delta_r\in\{-1,1\}
\]
such that
\[
\sigma(\rho_r)
=
\rho_{\pi_\sigma(r)}^{\delta_r}.
\]
Consequently, \(\sigma\) induces a bijection of \(\mathcal E_m\) given by
\[
\varepsilon'_{\pi_\sigma(r)}
=
\delta_r\varepsilon_r,
\]
and
\[
\sigma
\left(
\alpha_{\boldsymbol{\varepsilon}}
\right)
=
\alpha_{\boldsymbol{\varepsilon}'}.
\]
Thus, the indexed multiset \(\mathfrak A_S\) is Galois invariant.

It follows that every coefficient of \(\mathcal H_S(X)\) is rational.
These coefficients are also algebraic integers, being elementary symmetric
polynomials in algebraic integers.
Therefore,
\[
\mathcal H_S(X)\in\mathbb Z[X].
\]

Since \(|\mathcal E_m|=3^m\),
\[
\deg\mathcal H_S
=
3^m
=
3^{q-1}.
\]
Finally, equation \eqref{eq:Bn-expanded} expresses \(B_n(S)\) as a linear
combination of the exponential sequences
\(\alpha_{\boldsymbol{\varepsilon}}^n\).
Hence
\[
\mathcal H_S(E)B_n(S)=0
\]
for every \(n>2q\).
\end{proof}

\begin{remark}
\label{rem:H-independent-rho}
The polynomial \(\mathcal H_S(X)\) is independent of the ordering of the
roots \(\phi_r\) and of the choices of \(\rho_r\).
Indeed, permuting the roots or replacing any \(\rho_r\) by
\(\rho_r^{-1}\) only permutes the indexed multiset \(\mathfrak A_S\).

Repeated roots of \(Q_S(x)\), the exceptional value \(\phi_r=-2\), or
multiplicative relations among the numbers \(\rho_r\) may cause distinct
indices to produce the same exponential base.
Thus, \(\mathcal H_S(X)\) need not be square-free or minimal.
\end{remark}

\subsection{The minimal annihilating polynomial}
\label{subsec:minimal-annihilator}

For each \(\gamma\in\mathcal A_S\), define
\[
C_\gamma
=
\sum_{\substack{
\boldsymbol{\varepsilon}\in\mathcal E_m\\
\alpha_{\boldsymbol{\varepsilon}}=\gamma
}}
(-2)^{\kappa(\boldsymbol{\varepsilon})}.
\]
Collecting equal exponential bases in \eqref{eq:Bn-expanded}, we obtain
\begin{equation}
\label{eq:Bn-collected}
B_n(S)
=
\frac{1}{Q_S(2)}
\sum_{\gamma\in\mathcal A_S}
C_\gamma\gamma^n.
\end{equation}

\begin{theorem}
\label{thm:minimal-annihilating-polynomial}
Define
\[
\mathcal A_S^{*}
=
\left\{
\gamma\in\mathcal A_S:
C_\gamma\neq0
\right\}.
\]
Then the minimal monic annihilating polynomial of the tail sequence
\(\{B_n(S)\}_{n>2q}\) over \(\mathbb C\) is
\[
\mathcal M_S(X)
=
\prod_{\gamma\in\mathcal A_S^{*}}
(X-\gamma),
\]
and
\[
\mathcal M_S(X)\in\mathbb Z[X].
\]
Consequently,
\[
\deg\mathcal M_S
=
|\mathcal A_S^{*}|
\leq
|\mathcal A_S|
\leq
3^{q-1}.
\]

If the \(3^{q-1}\) indexed values
\(\alpha_{\boldsymbol{\varepsilon}}\) are pairwise distinct, then
\[
\mathcal M_S(X)
=
\mathcal H_S(X),
\]
and the minimal recurrence order is exactly
\[
3^{q-1}.
\]
If \(\mathcal A_S^{*}=\varnothing\), the empty product is understood to be
\(1\).
\end{theorem}

\begin{proof}
Equation \eqref{eq:Bn-collected} is a linear combination of exponential
sequences with distinct nonzero bases and nonzero coefficients indexed by
\(\mathcal A_S^{*}\).
Lemma \ref{lem:minimal-exponential-recurrence} therefore gives
\[
\mathcal M_S(X)
=
\prod_{\gamma\in\mathcal A_S^{*}}
(X-\gamma)
\]
as the minimal monic annihilating polynomial over \(\mathbb C\).

It remains to prove that its coefficients are integers.
As in the proof of Theorem
\ref{thm:integer-annihilating-polynomial}, every Galois automorphism
induces a permutation of \(\mathcal E_m\) that preserves
\(\kappa(\boldsymbol{\varepsilon})\) and maps
\(\alpha_{\boldsymbol{\varepsilon}}\) to its Galois conjugate.
Hence
\[
C_{\sigma(\gamma)}
=
C_\gamma.
\]
Thus, \(\mathcal A_S^{*}\) is Galois invariant, and
\(\mathcal M_S(X)\) has rational coefficients.
Since its roots are algebraic integers, these coefficients are also
algebraic integers.
Therefore,
\[
\mathcal M_S(X)\in\mathbb Z[X].
\]

If the indexed values
\(\alpha_{\boldsymbol{\varepsilon}}\) are pairwise distinct, then
\[
C_{\alpha_{\boldsymbol{\varepsilon}}}
=
(-2)^{\kappa(\boldsymbol{\varepsilon})}
\neq0.
\]
Hence every indexed base is active, so
\[
\mathcal M_S(X)=\mathcal H_S(X),
\]
and its degree is \(3^{q-1}\).
\end{proof}

\begin{remark}
\label{rem:order-reduction}
The recurrence order may be reduced in two ways.
First, distinct indices
\(\boldsymbol{\varepsilon}\) may determine the same exponential base.
This may result from repeated roots of \(Q_S(x)\), the exceptional value
\(\phi_r=-2\), or multiplicative relations among the numbers \(\rho_r\).
Second, after equal bases are collected, the coefficient \(C_\gamma\) may
vanish.
Thus, taking only the square-free part of
\(\mathcal H_S(X)\) need not give the minimal annihilating polynomial.
\end{remark}

\begin{remark}
\label{rem:spectral-function-role}
When \(G_n(S)\) is connected, Corollary
\ref{cor:Bn-interpretation} identifies \(B_n(S)\) with the normalized
spanning-tree number and the corresponding normalized special value of the
reduced spectral function.
Hence the recurrence for \(B_n(S)\) also applies to these quantities.
\end{remark}

\section{Concrete examples}
\label{sec:examples}

We apply the general results to four step sets.
The cases \(S=\{1\}\) and \(S=\{1,2\}\) are included as brief classical
examples.
For \(S=\{1,2,3\}\) and \(S=\{1,3\}\), we derive the corresponding
ninth-degree annihilating polynomials and prove their minimality.
Relevant Chebyshev formulas for these circulant graphs may be found in
\cite{ZhangYongGolin2005,MednykhMednykh2019}.

\subsection{The case \(S=\{1\}\)}
\label{subsec:S1}

\begin{proposition}
\label{prop:S1}
Let \(S=\{1\}\).
Then, for every \(n\geq3\),
\[
B_n(\{1\})=1.
\]
Consequently,
\[
M_n(\{1\})=n^2,
\qquad
\tau(C_n)=n,
\]
and
\[
\widetilde{\zeta}_{C_n}
\left(\frac{1}{2}\right)^{-1}
=
\frac{n^2}{2^{n-1}}.
\]
The minimal annihilating polynomial of
\(\{B_n(\{1\})\}_{n\geq3}\) is \(X-1\).
\end{proposition}

\begin{proof}
In this case,
\[
P_S(x)
=
2-2T_1\left(\frac{x}{2}\right)
=
2-x.
\]
Hence \(q=1\), \(Q_S(x)=1\), and the product over the roots of \(Q_S(x)\)
is empty.
Proposition \ref{thm:root-product} therefore gives
\[
B_n(\{1\})=1.
\]
The remaining identities follow from Definition \ref{def:Bn} and
Corollary \ref{cor:Bn-interpretation}.
Since the sequence is the nonzero constant sequence \(1\), its minimal
annihilating polynomial is \(X-1\).
\end{proof}

\subsection{The case \(S=\{1,2\}\)}
\label{subsec:S12}

\begin{proposition}
\label{prop:S12}
Let \(S=\{1,2\}\).
Then, for every \(n\geq5\),
\[
B_n(\{1,2\})=F_n^2.
\]
Consequently,
\[
M_n(\{1,2\})=n^2F_n^2,
\qquad
\tau(C_n^2)=nF_n^2,
\]
and
\[
\widetilde{\zeta}_{C_n^2}
\left(\frac{1}{4}\right)^{-1}
=
\frac{n^2F_n^2}{4^{n-1}}.
\]

The minimal annihilating polynomial of
\(\{B_n(\{1,2\})\}_{n\geq5}\) is
\[
\begin{aligned}
\mathcal M_{\{1,2\}}(X)
&=
(X+1)(X^2-3X+1)
\\
&=
X^3-2X^2-2X+1.
\end{aligned}
\]
Thus, for every \(n\geq5\),
\[
B_{n+3}(\{1,2\})
=
2B_{n+2}(\{1,2\})
+
2B_{n+1}(\{1,2\})
-
B_n(\{1,2\}).
\]
\end{proposition}

\begin{proof}
We have
\[
\begin{aligned}
P_S(x)
&=
4
-
2T_1\left(\frac{x}{2}\right)
-
2T_2\left(\frac{x}{2}\right)
\\
&=
6-x-x^2
\\
&=
(2-x)(x+3).
\end{aligned}
\]
Therefore,
\[
Q_S(x)=x+3,
\qquad
\phi=-3.
\]
Choose
\[
\rho=\frac{-3-\sqrt5}{2},
\]
so that
\[
\rho+\rho^{-1}=-3.
\]
Corollary \ref{thm:exponential-product} gives
\[
B_n(\{1,2\})
=
\frac{
(-\rho)^n+(-\rho^{-1})^n-2(-1)^n
}{5}.
\]
Since
\[
\rho^n+\rho^{-n}
=
(-1)^nL_{2n},
\]
where \(L_n\) denotes the \(n\)-th Lucas number, we obtain
\[
B_n(\{1,2\})
=
\frac{L_{2n}-2(-1)^n}{5}
=
F_n^2.
\]
This agrees with the classical formula of Kleitman and Golden
\cite{KleitmanGolden1975}.

The three exponential bases are
\[
-\rho,
\qquad
-\rho^{-1},
\qquad
-1,
\]
with respective coefficients
\[
\frac{1}{5},
\qquad
\frac{1}{5},
\qquad
-\frac{2}{5}.
\]
They are pairwise distinct, and all three coefficients are nonzero.
Theorem \ref{thm:minimal-annihilating-polynomial} therefore gives
\[
\begin{aligned}
\mathcal M_{\{1,2\}}(X)
&=
(X+\rho)(X+\rho^{-1})(X+1)
\\
&=
(X+1)(X^2-3X+1).
\end{aligned}
\]
The recurrence follows from the expanded polynomial.
\end{proof}

\subsection{The case \(S=\{1,2,3\}\)}
\label{subsec:S123}

\begin{proposition}
\label{prop:S123}
Let \(S=\{1,2,3\}\), and put
\[
b_n=B_n(\{1,2,3\})
\qquad
(n\geq7).
\]
Then the minimal annihilating polynomial of
\(\{b_n\}_{n\geq7}\) is
\[
\begin{aligned}
\mathcal M_{\{1,2,3\}}(X)
&=
(X-1)
(X^4-4X^3-X^2-4X+1)
\\
&\qquad\times
(X^4+3X^3+6X^2+3X+1)
\\
&=
X^9-2X^8-6X^7-21X^6
\\
&\qquad
+21X^3+6X^2+2X-1.
\end{aligned}
\]
Its minimal recurrence order is \(9\), and, for every \(n\geq7\),
\[
\begin{aligned}
b_{n+9}
&=
2b_{n+8}
+
6b_{n+7}
+
21b_{n+6}
\\
&\qquad
-
21b_{n+3}
-
6b_{n+2}
-
2b_{n+1}
+
b_n.
\end{aligned}
\]
Moreover,
\[
\tau(C_n^3)=nb_n
\]
and
\[
\widetilde{\zeta}_{C_n^3}
\left(\frac{1}{6}\right)^{-1}
=
\frac{n^2b_n}{6^{n-1}}.
\]
\end{proposition}

\begin{proof}
In this case,
\[
\begin{aligned}
P_S(x)
&=
6
-
2T_1\left(\frac{x}{2}\right)
-
2T_2\left(\frac{x}{2}\right)
-
2T_3\left(\frac{x}{2}\right)
\\
&=
8+2x-x^2-x^3
\\
&=
(2-x)(x^2+3x+4).
\end{aligned}
\]
Thus,
\[
Q_S(x)=x^2+3x+4,
\]
whose roots are
\[
\phi_{\pm}
=
\frac{-3\pm i\sqrt7}{2}.
\]
Choose \(\rho_{\pm}\) satisfying
\[
\rho_{\pm}+\rho_{\pm}^{-1}
=
\phi_{\pm}.
\]

The nine indexed exponential bases are
\[
\begin{gathered}
1,\quad
\rho_+,\quad
\rho_+^{-1},\quad
\rho_-,\quad
\rho_-^{-1},
\\
\rho_+\rho_-,
\quad
\rho_+\rho_-^{-1},
\quad
\rho_+^{-1}\rho_-,
\quad
\rho_+^{-1}\rho_-^{-1}.
\end{gathered}
\]

The four bases
\[
\rho_+,\quad
\rho_+^{-1},\quad
\rho_-,\quad
\rho_-^{-1}
\]
are the roots of
\[
\begin{aligned}
A(X)
&=
(X^2-\phi_+X+1)
(X^2-\phi_-X+1)
\\
&=
X^4
-
(\phi_++\phi_-)X^3
+
(\phi_+\phi_-+2)X^2
\\
&\qquad
-
(\phi_++\phi_-)X
+
1.
\end{aligned}
\]
Since
\[
\phi_++\phi_-=-3,
\qquad
\phi_+\phi_-=4,
\]
we obtain
\[
A(X)
=
X^4+3X^3+6X^2+3X+1.
\]

Next, put
\[
u=\rho_+\rho_-,
\qquad
v=\rho_+\rho_-^{-1},
\]
and
\[
a=u+u^{-1},
\qquad
c=v+v^{-1}.
\]
The remaining four bases are
\[
u,\quad u^{-1},\quad v,\quad v^{-1}.
\]
We have
\[
\begin{aligned}
a+c
&=
(\rho_++\rho_+^{-1})
(\rho_-+\rho_-^{-1})
\\
&=
\phi_+\phi_-
\\
&=
4
\end{aligned}
\]
and
\[
\begin{aligned}
ac
&=
\rho_+^2+\rho_+^{-2}
+
\rho_-^2+\rho_-^{-2}
\\
&=
(\phi_+^2-2)+(\phi_-^2-2)
\\
&=
(\phi_++\phi_-)^2
-
2\phi_+\phi_-
-
4
\\
&=
-3.
\end{aligned}
\]
Hence these four bases are the roots of
\[
\begin{aligned}
D(X)
&=
(X^2-aX+1)(X^2-cX+1)
\\
&=
X^4-(a+c)X^3+(ac+2)X^2-(a+c)X+1
\\
&=
X^4-4X^3-X^2-4X+1.
\end{aligned}
\]
Therefore,
\[
\mathcal H_{\{1,2,3\}}(X)
=
(X-1)D(X)A(X).
\]

It remains to prove minimality.
Since
\[
\phi_+\neq\phi_-,
\qquad
\phi_{\pm}\neq\pm2,
\]
the polynomial \(A(X)\) is square-free.
Similarly, \(a\) and \(c\) are the distinct roots of
\[
t^2-4t-3=0,
\]
and neither is equal to \(\pm2\), so \(D(X)\) is square-free.
Moreover,
\[
A(1)=14,
\qquad
D(1)=-7,
\]
and hence neither quartic factor has \(1\) as a root.

Finally,
\[
A(X)-D(X)
=
7X(X^2+X+1).
\]
Suppose that \(z\) is a common root of \(A(X)\) and \(D(X)\).
Since both polynomials have constant term \(1\), we have \(z\neq0\), and
therefore
\[
z^2+z+1=0.
\]
Using
\[
z^2=-z-1,
\qquad
z^3=1,
\]
we obtain
\[
A(z)=-2(z+1)\neq0,
\]
a contradiction.
Thus,
\[
\gcd(A,D)=1.
\]
It follows that the nine indexed bases are pairwise distinct.

Since
\[
Q_S(2)=14,
\]
the coefficient of every indexed base in \eqref{eq:Bn-expanded} is of the
form
\[
\frac{(-2)^k}{14},
\qquad
k=0,1,2,
\]
and is therefore nonzero.
Theorem \ref{thm:minimal-annihilating-polynomial} gives
\[
\mathcal M_{\{1,2,3\}}(X)
=
\mathcal H_{\{1,2,3\}}(X).
\]
Thus, the displayed polynomial is minimal and has degree \(9\).
The graph-theoretic identities follow from Corollary
\ref{cor:Bn-interpretation}.
\end{proof}

\subsection{The case \(S=\{1,3\}\)}
\label{subsec:S13}

\begin{proposition}
\label{prop:S13}
Let \(S=\{1,3\}\), and put
\[
b_n=B_n(\{1,3\})
\qquad
(n\geq7).
\]
Then
\[
b_n
=
\frac{2}{5}
\left\{
T_n\left(\frac{-1+i}{2}\right)-1
\right\}
\left\{
T_n\left(\frac{-1-i}{2}\right)-1
\right\}.
\]

The minimal annihilating polynomial of
\(\{b_n\}_{n\geq7}\) is
\[
\begin{aligned}
\mathcal M_{\{1,3\}}(X)
&=
(X-1)
(X^4-2X^3-2X^2-2X+1)
\\
&\qquad\times
(X^4+2X^3+4X^2+2X+1)
\\
&=
X^9-X^8-2X^7-10X^6-2X^5
\\
&\qquad
+2X^4+10X^3+2X^2+X-1.
\end{aligned}
\]
Its minimal recurrence order is \(9\), and, for every \(n\geq7\),
\[
\begin{aligned}
b_{n+9}
&=
b_{n+8}
+
2b_{n+7}
+
10b_{n+6}
+
2b_{n+5}
\\
&\qquad
-
2b_{n+4}
-
10b_{n+3}
-
2b_{n+2}
-
b_{n+1}
+
b_n.
\end{aligned}
\]
Moreover,
\[
\tau(G_n(\{1,3\}))=nb_n
\]
and
\[
\widetilde{\zeta}_{G_n(\{1,3\})}
\left(\frac{1}{4}\right)^{-1}
=
\frac{n^2b_n}{4^{n-1}}.
\]
\end{proposition}

\begin{proof}
We have
\[
\begin{aligned}
P_S(x)
&=
4
-
2T_1\left(\frac{x}{2}\right)
-
2T_3\left(\frac{x}{2}\right)
\\
&=
4-x-(x^3-3x)
\\
&=
4+2x-x^3
\\
&=
(2-x)(x^2+2x+2).
\end{aligned}
\]
Therefore,
\[
Q_S(x)=x^2+2x+2,
\]
whose roots are
\[
\phi_{\pm}=-1\pm i.
\]
Since
\[
Q_S(2)=10,
\]
Proposition \ref{thm:root-product} gives
\[
\begin{aligned}
b_n
&=
\frac{
\left\{
2T_n\left(\frac{\phi_+}{2}\right)-2
\right\}
\left\{
2T_n\left(\frac{\phi_-}{2}\right)-2
\right\}
}{10}
\\
&=
\frac{2}{5}
\left\{
T_n\left(\frac{-1+i}{2}\right)-1
\right\}
\left\{
T_n\left(\frac{-1-i}{2}\right)-1
\right\}.
\end{aligned}
\]

Choose \(\rho_{\pm}\) satisfying
\[
\rho_{\pm}+\rho_{\pm}^{-1}
=
\phi_{\pm}.
\]
The nine indexed exponential bases are
\[
\begin{gathered}
1,\quad
\rho_+,\quad
\rho_+^{-1},\quad
\rho_-,\quad
\rho_-^{-1},
\\
\rho_+\rho_-,
\quad
\rho_+\rho_-^{-1},
\quad
\rho_+^{-1}\rho_-,
\quad
\rho_+^{-1}\rho_-^{-1}.
\end{gathered}
\]

The four bases
\[
\rho_+,\quad
\rho_+^{-1},\quad
\rho_-,\quad
\rho_-^{-1}
\]
are the roots of
\[
\begin{aligned}
A(X)
&=
(X^2-\phi_+X+1)
(X^2-\phi_-X+1)
\\
&=
X^4
-
(\phi_++\phi_-)X^3
+
(\phi_+\phi_-+2)X^2
\\
&\qquad
-
(\phi_++\phi_-)X
+
1
\\
&=
X^4+2X^3+4X^2+2X+1,
\end{aligned}
\]
where
\[
\phi_++\phi_-=-2,
\qquad
\phi_+\phi_-=2.
\]

Put
\[
u=\rho_+\rho_-,
\qquad
v=\rho_+\rho_-^{-1},
\]
and
\[
a=u+u^{-1},
\qquad
c=v+v^{-1}.
\]
Then
\[
a+c
=
\phi_+\phi_-
=
2
\]
and
\[
\begin{aligned}
ac
&=
(\phi_+^2-2)+(\phi_-^2-2)
\\
&=
(\phi_++\phi_-)^2
-
2\phi_+\phi_-
-
4
\\
&=
-4.
\end{aligned}
\]
Therefore, the remaining four bases are the roots of
\[
\begin{aligned}
D(X)
&=
(X^2-aX+1)(X^2-cX+1)
\\
&=
X^4-(a+c)X^3+(ac+2)X^2-(a+c)X+1
\\
&=
X^4-2X^3-2X^2-2X+1.
\end{aligned}
\]
Hence
\[
\mathcal H_{\{1,3\}}(X)
=
(X-1)D(X)A(X).
\]

To prove minimality, first note that \(A(X)\) is square-free because
\[
\phi_+\neq\phi_-,
\qquad
\phi_{\pm}\neq\pm2.
\]
Similarly, \(a\) and \(c\) are the distinct roots of
\[
t^2-2t-4=0
\]
and neither is equal to \(\pm2\), so \(D(X)\) is square-free.
Moreover,
\[
A(1)=10,
\qquad
D(1)=-4.
\]

Finally,
\[
A(X)-D(X)
=
2X(2X^2+3X+2).
\]
Suppose that \(z\) is a common root of \(A(X)\) and \(D(X)\).
Then \(z\neq0\) and
\[
2z^2+3z+2=0.
\]
Reducing \(8A(z)\) modulo this relation gives
\[
8A(z)=-5(3z+2).
\]
Thus \(A(z)=0\) would imply
\[
z=-\frac{2}{3},
\]
but
\[
2\left(-\frac{2}{3}\right)^2
+
3\left(-\frac{2}{3}\right)
+
2
=
\frac{8}{9}
\neq0.
\]
This contradiction proves that
\[
\gcd(A,D)=1.
\]
Hence the nine indexed bases are pairwise distinct.

Since
\[
Q_S(2)=10,
\]
the coefficient of every indexed base in \eqref{eq:Bn-expanded} is of the
form
\[
\frac{(-2)^k}{10},
\qquad
k=0,1,2,
\]
and is nonzero.
Theorem \ref{thm:minimal-annihilating-polynomial} therefore gives
\[
\mathcal M_{\{1,3\}}(X)
=
\mathcal H_{\{1,3\}}(X).
\]
Thus, the displayed polynomial is minimal and has degree \(9\).
The graph-theoretic identities follow from Corollary
\ref{cor:Bn-interpretation}.
\end{proof}

\section{Exact algebraic verification}
\label{sec:verification}

We verify the formulas for
\[
S=\{1,3\}
\qquad\text{and}\qquad
S=\{1,2,3\}
\]
by two exact calculations.
No numerical approximation of the algebraic roots is used.

Define
\[
\mathcal C_n(x)
=
2T_n\left(\frac{x}{2}\right).
\]
For every \(n\geq1\), the polynomial \(\mathcal C_n(x)\) is monic and
belongs to \(\mathbb Z[x]\).
Moreover,
\[
\mathcal C_0(x)=2,
\qquad
\mathcal C_1(x)=x,
\qquad
\mathcal C_{n+1}(x)
=
x\mathcal C_n(x)-\mathcal C_{n-1}(x).
\]

Since \(Q_S(x)\) is monic with roots
\(\phi_1,\ldots,\phi_{q-1}\), counted with multiplicities,
\[
\operatorname{Res}_x
\left(
Q_S(x),\mathcal C_n(x)-2
\right)
=
\prod_{r=1}^{q-1}
\left\{
\mathcal C_n(\phi_r)-2
\right\}.
\]
Proposition \ref{thm:root-product} therefore gives
\[
B_n^{\mathrm{root}}(S)
=
\frac{(-1)^{n(q-1)}}{Q_S(2)}
\operatorname{Res}_x
\left(
Q_S(x),\mathcal C_n(x)-2
\right).
\]
The resultant is evaluated as the determinant of the Sylvester matrix of
\[
Q_S(x)
\qquad\text{and}\qquad
\mathcal C_n(x)-2,
\]
using integer polynomial arithmetic.
For the second calculation, let \(L_n(S)\) denote the Laplacian matrix of
\(G_n(S)\), and let \(L_n(S)^{(0)}\) be the
\((n-1)\times(n-1)\) principal submatrix obtained by deleting one row
and the corresponding column.
Since both step sets contain \(1\), the corresponding graphs are connected.
Kirchhoff's Matrix--Tree Theorem gives
\[
B_n^{\mathrm{det}}(S)
=
\frac{1}{n}
\det L_n(S)^{(0)}.
\]
The determinants were evaluated over the integers using the Bareiss
fraction-free elimination algorithm.

\begin{table}[htbp]
\centering
\caption{Exact verification for \(S=\{1,3\}\).}
\label{tab:S13-verification}
\begin{tabular}{c|r|r}
\hline
\(n\)
&
\(B_n^{\mathrm{root}}(\{1,3\})\)
&
\(B_n^{\mathrm{det}}(\{1,3\})\)
\\
\hline
\(7\)  & \(169\)       & \(169\)       \\
\(8\)  & \(512\)       & \(512\)       \\
\(9\)  & \(1369\)      & \(1369\)      \\
\(10\) & \(4050\)      & \(4050\)      \\
\(11\) & \(11881\)     & \(11881\)     \\
\(12\) & \(33800\)     & \(33800\)     \\
\(13\) & \(97969\)     & \(97969\)     \\
\(14\) & \(284258\)    & \(284258\)    \\
\(15\) & \(819025\)    & \(819025\)    \\
\(16\) & \(2367488\)   & \(2367488\)   \\
\hline
\end{tabular}
\end{table}

\begin{table}[htbp]
\centering
\caption{Exact verification for \(S=\{1,2,3\}\).}
\label{tab:S123-verification}
\begin{tabular}{c|r|r}
\hline
\(n\)
&
\(B_n^{\mathrm{root}}(\{1,2,3\})\)
&
\(B_n^{\mathrm{det}}(\{1,2,3\})\)
\\
\hline
\(7\)  & \(2401\)        & \(2401\)        \\
\(8\)  & \(10368\)       & \(10368\)       \\
\(9\)  & \(45796\)       & \(45796\)       \\
\(10\) & \(203522\)      & \(203522\)      \\
\(11\) & \(896809\)      & \(896809\)      \\
\(12\) & \(3964928\)     & \(3964928\)     \\
\(13\) & \(17530969\)    & \(17530969\)    \\
\(14\) & \(77451458\)    & \(77451458\)    \\
\(15\) & \(342324004\)   & \(342324004\)   \\
\(16\) & \(1512940032\)  & \(1512940032\)  \\
\hline
\end{tabular}
\end{table}

Tables \ref{tab:S13-verification} and
\ref{tab:S123-verification} show agreement between the resultant and
Laplacian cofactor calculations for ten consecutive values.
Since the recurrences in Propositions \ref{prop:S13} and
\ref{prop:S123} have order \(9\), the values corresponding to
\[
n=7,8,\ldots,15
\]
form a complete set of initial values.
In both cases, the recurrence with \(n=7\) gives the displayed value at
\(n=16\), in agreement with both exact calculations.

\section{Conclusion and future problems}
\label{sec:conclusion}

In this paper, we studied the normalized algebraic product sequence
\[
B_n(S)
=
\frac{1}{n^2}
\prod_{j=1}^{n-1}
P_S\left(
2\cos\frac{2\pi j}{n}
\right),
\qquad
n>2\max S,
\]
associated with a fixed nonempty finite set \(S\) of positive integers.

Let \(q=\max S\).
Starting from the Chebyshev root representation of \(B_n(S)\), we
constructed an explicit monic annihilating polynomial
\[
\mathcal H_S(X)\in\mathbb Z[X]
\]
of degree
\[
3^{q-1}.
\]
Hence the tail sequence
\[
\{B_n(S)\}_{n>2q}
\]
satisfies an integer-coefficient linear recurrence of order at most
\(3^{q-1}\).

We also determined the minimal annihilating polynomial.
After equal exponential bases are collected, let
\(\mathcal A_S^{*}\) denote the set of distinct bases whose total
coefficients are nonzero.
Then
\[
\mathcal M_S(X)
=
\prod_{\gamma\in\mathcal A_S^{*}}
(X-\gamma)
\in\mathbb Z[X],
\]
and the minimal recurrence order is
\[
|\mathcal A_S^{*}|.
\]
This description takes account of both coincidences among exponential
bases and cancellations among their coefficients.
In particular, if all \(3^{q-1}\) indexed bases are pairwise distinct,
then
\[
\mathcal M_S(X)=\mathcal H_S(X),
\]
and the upper bound \(3^{q-1}\) is attained.

For \(S=\{1,2,3\}\) and \(S=\{1,3\}\), we derived the corresponding
ninth-degree minimal annihilating polynomials explicitly.
In both cases, the nine indexed bases are pairwise distinct and have
nonzero coefficients.
The formulas and recurrences were also checked by exact resultant and
Laplacian cofactor calculations.

When \(G_n(S)\) is connected,
\[
B_n(S)
=
\frac{\tau(G_n(S))}{n}
=
\frac{(2|S|)^{n-1}}{n^2}
\widetilde{\zeta}_{G_n(S)}
\left(
\frac{1}{2|S|}
\right)^{-1}.
\]
Thus, the recurrence obtained for the algebraic sequence \(B_n(S)\)
also applies to the normalized spanning-tree number and the corresponding
special value of the reduced spectral function.

The present method relies on a fixed step set \(S\), the restriction
\(n>2\max S\), and the one-variable Chebyshev factorization arising from
the cyclic group.
Moreover, the bound \(3^{q-1}\) need not be minimal when exponential bases
coincide or their coefficients cancel.

A natural problem is to characterize the step sets \(S\) for which the
upper bound \(3^{q-1}\) is attained.
More generally, it would be useful to describe reductions in the minimal
recurrence order directly in terms of algebraic and multiplicative
relations among the roots of \(Q_S(x)\).

Another direction is to study the dominant exponential bases and their
relation to the asymptotic behavior of \(B_n(S)\).
It would also be interesting to extend the integer-annihilating-polynomial
construction to more general Abelian Cayley graphs.
Such an extension would require suitable multivariable analogues of the
Chebyshev factorization used here.

\section*{Acknowledgments}

We want to thank the referee very much for their valuable comments.


\end{document}